\documentclass{amsart}

\usepackage{amsmath}
\usepackage{amsfonts}
\usepackage{amstext}
\usepackage{amsbsy}
\usepackage{amsopn}
\usepackage{amsxtra}
\usepackage{upref}
\usepackage{amsthm}
\usepackage{amsmath}
\usepackage{amssymb}
\usepackage{epsfig,verbatim,ifthen,graphicx}
\newtheorem{prop}{Proposition}[section]
\newtheorem{lema}{Lemma}[section]
\newtheorem{teo}{Theorem}[section]

\newtheorem{coro}{Corollary}[section]

\theoremstyle{definition}
\newtheorem{defi}{Definition}[section]
\newtheorem{rem}{Remark}[section]

\def\R{{\mathbb R}}

\def\N{{\mathbb N}}

\def\M{{\mathcal M}}

\def\var{{\text{var}}}

\title[Weak Gibbs and asymptotically additive sequences]{Weak Gibbs measures as Gibbs measures for asymptotically additive sequences}
\date{\today}

\begin{thanks}
{Both authors were supported by  the Center of Dynamical Systems and Related Fields c\'odigo ACT1103  PIA - Conicyt. GI
was partially  supported by Proyecto Fondecyt 1150058 and YY was supported by Proyecto Fondecyt 1151368. }
\end{thanks}


\author{Godofredo Iommi} \address{Facultad de Matem\'aticas,
Pontificia Universidad Cat\'olica de Chile (PUC), Avenida Vicu\~na Mackenna 4860, Santiago, Chile}
\email{giommi@mat.puc.cl}
\urladdr{http://www.mat.puc.cl/\textasciitilde giommi/}
\author{Yuki Yayama}
\address{Departamento de Ciencias B\'{a}sicas, Universidad del B\'{i}o-B\'{i}o, Avenida Andr\'{e}s Bello, s/n
Casilla 447, Chill\'{a}n, Chile}
\email{yyayama@ubiobio.cl}

\begin{document}

\begin{abstract}
In this note we prove that every weak Gibbs measure for an asymptotically additive sequences is a Gibbs measure for
another  asymptotically additive sequence. In particular, a weak Gibbs measure for a continuous potential is
a Gibbs measure for an asymptotically additive sequence. This allows, for example, to apply recent results on dimension
theory of asymptotically additive sequences to  study multifractal analysis for weak Gibbs measure for continuous potentials.
\end{abstract}

\maketitle

\section{Introduction and Preliminaries}
Gibbs measures have played a prominent role in ergodic theory since the definition was brought from statistical
mechanics into dynamical systems (see for example \cite{bow, sin}).  Existence of Gibbs measures usually requires strong forms of hyperbolicity on the system and of regularity on the potential. In her study of equilibrium measures for intermittent interval maps, Yuri \cite{y1,y2,y3} introduced the notion of weak Gibbs measure.  It turns out that these measures, which generalize the classical notion of Gibbs measures, exist under meagre regularity assumptions on the potential and for a wider class of dynamical systems. Technically, the main difference between the two notions is that for Gibbs measures we have a uniform control on the measure of dynamical balls while for weak Gibbs measures this control is not uniform (see Section \ref{s.wg} for precise statements).  This circle of ideas has been generalized to settings in which instead of considering a single potential we consider a sequence of potentials. This theory, usually called non-additive thermodynamic formalism, was introduced by Falconer \cite{f} with the purpose of studying dimension theory of non-conformal systems. We stress that it is also a well suited theory to study products of matrices. The purpose of the present note is to show that a weak Gibbs measure for a potential is actually a Gibbs measure for a sequence of potentials.  That point of view allows for the use of  machinery developed to study Gibbs measures for sequences of potentials in the study of weak Gibbs measures for a single potential. As an example of the possibilities that are opened with this viewpoint we prove a variational principle for higher order multifractal analysis of weak Gibbs measures.

\subsection{Weak Gibbs measures} \label{s.wg}
This section is devoted to define the notion of weak Gibbs measure for compact symbolic spaces and to briefly discuss conditions that ensure both, their existence and that they do not have atoms.

Let $(\Sigma, \sigma)$ be a one-sided Markov shift defined over a finite alphabet $\mathcal{A}$. This means that there exists a matrix $T=(t_{ij})_{\mathcal{A} \times \mathcal{A} }$ of zeros and ones such that
\begin{equation*}
\Sigma=\left\{ \omega \in \mathcal{A} ^{\N} : t_{\omega_{i} \omega_{i+1}}=1 \ \text{for every $i \in \N$}\right\}.
\end{equation*}
The \emph{shift map} $\sigma:\Sigma \to \Sigma$ is defined by $\sigma(\omega_1 \omega_2 \dots)=(\omega_2 \omega_3 \dots)$.  We assume that  $(\Sigma, \sigma)$ is topologically mixing, which in this setting means that there exists a positive integer $l \in \N$ such that
all the entries of the matrix $T^l$ are strictly positive. The set
\begin{equation*}
 C_{i_1 \cdots i_{n}}:= \left\{\omega \in \Sigma : \omega_j=i_j \text{ for } 1 \le j \le n \right\}
 \end{equation*}
is called \emph{cylinder} of length $n$. The space $\Sigma$ endowed with the topology generated by cylinder sets is a compact space. Morevoer,  the metric $d:\Sigma \times \Sigma \to \R $ defined by $d(\omega, \kappa):=\sum_{i=1}^{\infty} \frac{|\omega_i - \kappa_i|}{2^i}$ induces the same topology as the one generated by the cylinder sets.

 In order to understand the geometric or dynamical  properties of a probability measure $\mu$ on $\Sigma$  it is of great importance to have good estimates on the measure of the cylinder sets. This simple remark is captured in the definition of Gibbs measure (see \cite[Chapter 1]{bow}) which was later generalized by  Yuri  \cite{y1,y2,y3} in the following sense.

\begin{defi} \label{d.wg}
A probability measure $\mu$ is called \emph{weak Gibbs} measure for the potential $\phi: \Sigma \to \R$ if there exists $P\in \R$ and  a
 sequence of positive real numbers $K(n)$ satisfying
\[\lim_{n \to \infty} \frac{ \log K(n)}{n}=0,\]
such that, for every $n \in \N$, every cylinder $C_{i_1 \dots i_n}$ and every  $\omega \in
C_{i_1 \dots i_n}$ we have
\begin{equation*}
\frac{1}{K(n)} \leq \frac{\mu(C_{i_1 \dots i_{n}})}{\exp(S_n \phi(\omega) -n P)}\leq K(n),
\end{equation*}
where $S_n\phi(\omega):= \sum_{j=0}^{n-1} \phi(\sigma^j \omega)$.
\end{defi}
If for every $n \in \N$ we have that $K(n)=K$ we recover the classical notion of Gibbs measure (see \cite[Chapter 1]{bow}).
It turns out that  weak Gibbs measures exists under very mild assumptions on the potential and that the number $P$ can be characterized in thermodynamic terms.  Let $\phi:\Sigma \to \R$, we define its \emph{n-th variation} by
\[ \var_n(\phi) := \sup_{C \in Z_n} \sup_{\omega,\kappa \in C} \{|\phi(\omega) -\phi(\kappa)|  \}, \]
where $Z_n$ denotes the set of all cylinders of length $n$. Denote by $\eta_{\phi}(n) := \var_n (S_n \phi)$. The class of \emph{medium varying functions} is that of bounded, measurable, real-valued functions $\phi:\Sigma \to \R$ such that
\[ \lim_{n \to \infty} \frac{\eta_{\phi}(n)}{n} =0.\]
Note that every continuous function belongs to this class.  Let $\phi:\Sigma \to \R$ be a medium varying function,
the \emph{topological pressure} of $\phi$ is defined by
$$P(\phi):= \lim_{n\rightarrow\infty} \frac{1}{n} \log \sum_{C \in Z_n} \exp \left(\sup_{\omega\in C}S_n\phi(\omega)) \right).$$
Note that the above limit exists and is finite (see \cite[Lemma 1]{k}). The following result due to Kesseb\"ohmer \cite[Section 2]{k} establishes the existence of weak Gibbs measure for medium varying functions.

\begin{prop}[Kesseb\"ohmer]\label{bound}
For every medium varying function $\phi:\Sigma \to \R$ there exists a weak Gibbs measure $\mu$. Moreover, in definition \ref{d.wg}, the constants can be chosen so that
\[K(n)= \exp(\eta_{\phi}(n)),\]
and the value of $P$ is equal to $P(\phi)$.
\end{prop}

\begin{rem} \label{non-ato}
Kesseb\"ohmer also showed that a weak Gibbs measure $\mu$ for the potential $\phi$ is positive on cylinders and if in addition we have that $\sup \phi < P(\phi)$ then the measure $\mu$ is atom free. Inoquio-Renteria and Rivera-Letelier in \cite[Proposition 3.1]{ij} showed that if $\phi:\Sigma \to \R$ is a continuous potential
satisfying the following condition: there exists an integer $n \geq 1$ such that
\begin{equation} \label{juan}
\sup_{\omega \in \Sigma} \frac{1}{n} S_n(\phi(\omega)) < P(\phi),
\end{equation}
then $\phi$ is cohomologous to a continuous potential $\overline{\phi}$ satisfying
\[ \sup_{\omega \in \Sigma} \overline{\phi}(\omega) < P(\overline{\phi}).\]
In particular, if a continuous potential $\phi$ satisfies equation \eqref{juan} then the weak Gibbs measure for $\phi$ is non-atomic. Let us stress that there exists continuous potentials satisfying equation \eqref{juan} for some $n >1$ and not for $n=1$ (see \cite[Remark 3.3]{ij}).
 \end{rem}

Kesseb\"ohmer \cite{k} in fact proved that conformal measures corresponding to medium varying function are indeed weak Gibbs measures.
In general these are not invariant measures.  However, if the potential $\phi$ is a  g-function (see \cite{ke} for a precise
definition) then the conformal measure is invariant. Also, if we assume more regularity on the potential $\phi$,   for example potentials belonging to the Bowen class (see \cite[p.329]{w2} for a precise definition)
 then the potential  is cohomologous to a g-function (see  \cite[Proof of Theorem 4.5]{w2}) and thus  has an invariant weak Gibbs measure.

\subsection{Non-additive thermodynamic formalism} \label{s.na} Motivated from the study of dimension theory of non-conformal dynamical systems, Falconer \cite{f} introduced a non-additive generalization of thermodynamic formalism. This theory has been considerably developed over the last years (see \cite{b5} for a general account on the theory). In this non-additive setting, the topological pressure of a continuous function is replaced by the topological pressure of a sequence of continuous functions. Different types of additivity assumptions are usually made on the sequence. For example, we say that a sequence $\Phi:=(\phi_n)_n$ of continuous    functions $\phi_n :\Sigma \to\R$ is \emph{almost additive} if there exists a constant $C >0$ such that for every $n \in \N$ and every $\omega \in \Sigma$ we have
\begin{equation*}
 -C+ \phi_n(\omega) + \phi_{m}(\sigma^n \omega) \leq   \phi_{n+m}(\omega) \leq C+ \phi_n(\omega) + \phi_{m}(\sigma^n \omega).
\end{equation*}
Under this assumption thermodynamic formalism has been thoroughly studied (see \cite{b3,iy,m}). Another type of additivity assumption was introduced by Feng and Huang \cite{fh},
a sequence of continuous functions $\Phi:=(\phi_n)_n$ is
\emph{asymptotically additive} on $\Sigma$ if for every $\epsilon >0$ there exists a continuous function $\rho_{\epsilon}$ such that
\begin{equation}\label{aymad}
\limsup_{n \to \infty} \frac{1}{n} \| \phi_n -S_n \rho_{\epsilon}  \|  < \epsilon,
\end{equation}
where $\| \cdot \|$ is the supremum norm. Every almost additive sequence is asymptotically additive (see,
for example, \cite[Proposition 2.1]{zzc}) but the converse is not true.

Building upon the work of Feng and Huang  \cite{fh}, the topological pressure in this setting was defined by Cheng, Zhao and  Cao \cite{czc} (see also \cite{cfh, zzc} for related results). Let $\Phi=(\phi_n)_{n}$ be an asymptotically additive sequence on a sub-shift $\Sigma$.  Let $n \in \N$ and $\epsilon>0$,
a subset $E \subset \Sigma$ is an {\em $(n, \epsilon)$-separated subset} of
$\Sigma$ if $\max_{0\leq i\leq n-1}d(\sigma^{i}\omega, \sigma^{i}\kappa)>\epsilon$ for all $\omega, \kappa \in E, \omega\neq \kappa$.
Let
\begin{equation*}
P(\Phi, n,\epsilon):=\sup\left\{\sum_{\omega\in E}\exp(\phi_n(\omega)): E \text { is an }
(n, \epsilon)\text{-separated subset of } \Sigma \right\}.
\end{equation*}
The \emph{topological pressure} for
an asymptotically additive sequence $\Phi$ on $\Sigma$ is defined by
\begin{equation}\label{deftop}
P(\Phi):=\lim_{\epsilon\rightarrow  0}
\limsup_{n\rightarrow \infty} \frac{1}{n} \log P(\Phi, n, \epsilon).
\end{equation}
This notion of pressure satisfies the corresponding variational principle (see \cite{fh}),
\[P(\Phi)= \sup \left\{ h(\mu) + \lim_{n \to \infty} \frac{1}{n} \int \phi_n d \mu: \mu \in \M\right\},\]
where  $\M$ denotes the set of $\sigma-$invariant probability measures
and $h(\mu)$  the entropy of the measure $\mu$ (see \cite[Chapter 4]{w}).

In what follows we will give another characterization of the pressure using periodic points instead of  $(n, \epsilon)$-separated sets. This characterization is well known for the pressure of a single function defined on a topologically mixing  sub-shift of finite type. Let $\text{Fix}(\sigma):=\{ \omega\in \Sigma: \sigma \omega=\omega \}$. If $\phi:\Sigma \to \R$ is a continuous function and $(\Sigma, \sigma)$ is topologically mixing, it was shown by Ruelle (see \cite[Theorem 7.20]{ru}) that
\begin{equation} \label{ruelle}
P(\phi) =\lim_{n\rightarrow\infty} \frac{1}{n}\log \sum_{\omega \in \text{Fix}(\sigma^n)} \exp(S_n(\phi_n(\omega)).
\end{equation}
We further generalize this result to setting of asymptotically additive sequences.

\begin{prop}\label{defpressure}
Let $\Phi=(\phi_n)_{n}$ be an asymptotically additive sequence on a topologically mixing shift of finite type
 $(\Sigma, \sigma)$. Then
\begin{equation}\label{imp1}
P(\Phi) =\lim_{n\rightarrow\infty} \frac{1}{n}\log \sum_{\omega \in \text{Fix}(\sigma^n)} \exp(\phi_n(\omega)).
\end{equation}
\end{prop}

\begin{rem} \label{reg}
Let us note that  Proposition \ref{defpressure} has been proved under stronger additivity and regularity  assumptions.  Indeed,  let $\Phi=(\phi_n)_n$ be a sequence of continuous functions on $\Sigma$ and
\begin{equation}
\gamma_{n}(\Phi) := \var_n \phi_n= \sup\{\vert \phi_n(\omega)-\phi_n(\kappa)\vert: \omega, \kappa \in C_{i_1\dots i_n}\}.
\end{equation}
We say that $\Phi$ has \emph{tempered variation} if $\lim_{n\rightarrow \infty}\gamma_{n}(\Phi)/n=0$.
Barreira \cite{b2} showed  that if  $\Phi=(\phi_n)_n$ is an   almost additive sequences with tempered variation the equation (\ref{imp1}) holds. It should be noted, however, that  a simple modification of the argument in \cite[Lemma 2.1]{zzc} implies that  if $\Phi=(\phi_n)_{n}$ is an asymptotically additive sequence on a topologically mixing sub-shift  $(\Sigma, \sigma)$ then $\Phi$ has tempered variation.
\end{rem}
%
%
%

In order to prove proposition \ref{defpressure} we require some auxiliary results. A version of the following Lemma was proved in \cite[Lemma 2.3]{zzc}.

 \begin{lema}\label{approx}
Let $\Phi=(\phi_n)_{n}$ be an asymptotically additive sequence on a topologically mixing sub-shift $(\Sigma, \sigma)$
satisfying equation (\ref{aymad}). Then
$P(\Phi)=\lim_{\epsilon\rightarrow 0}P(\rho_{\epsilon})$.
\end{lema}

\begin{proof}
Let $\bar\epsilon>0$.  Since $\Phi$ is an asymptotically additive sequence, there exist
$\rho_{\bar \epsilon}\in C(\Sigma)$ and $N_{\bar\epsilon}$ such that  for every $ \omega\in \Sigma$ and $n\geq N_{\bar\epsilon}$
\begin{equation*}
\frac{1}{n}(S_n\rho_{\bar \epsilon})(\omega)-2\bar \epsilon \leq \frac{1}{n}\phi_n (\omega)
\leq \frac{1}{n}(S_n\rho_{\bar \epsilon})(\omega)+2\bar \epsilon.
\end{equation*}
Thus, for $n\geq N_{\bar\epsilon}$ and $\epsilon>0$
\begin{align*}
&\exp({-2n}{\bar \epsilon})\sup\left\{\sum_{\omega\in E}
\exp((S_n\rho_{\bar \epsilon}) (\omega)): E \text{ is an } (n, \epsilon)
\text{ separated subset of } \Sigma \right\}\\
&\leq \sup\left\{\sum_{\omega\in E} \exp(\phi_n(\omega)): E \text{ is an }(n, \epsilon)
\text{ separated subset of }
\Sigma\right\}\\
&\leq
\exp(2{n}{\bar \epsilon})\sup\left\{\sum_{\omega \in E} \exp((S_n\rho_{\bar \epsilon}) (\omega)): E \text{ is an } (n, \epsilon)
\text{ separated subset of } \Sigma \right\}.
\end{align*}
Therefore,
\begin{align*}
&-2\bar \epsilon +\frac{1}{n}
\log \left(\sup\left\{\sum_{\omega\in E} \exp((S_n\rho_{\bar \epsilon})(\omega)): E \text{ is an }(n, \epsilon)
\text{ separated subset of } \Sigma\right\}\right)\\
&\leq \frac{1}{n}\log \left(\sup\left\{\sum_{\omega \in E}\exp(\phi_n(\omega)):E \text{ is an }(n, \epsilon)\text{
separated subset of } \Sigma\right\}\right)\\
&\leq 2\bar \epsilon +\frac{1}{n}
\log \left(\sup\left\{\sum_{\omega \in E} \exp((S_n\rho_{\bar \epsilon})(\omega)): E \text{ is an }(n, \epsilon)\text{
separated subset of } \Sigma\right\}\right).
\end{align*}
Hence
\begin{equation*}
-2\bar\epsilon +P(\rho_{\bar\epsilon})\leq P(\Phi)\leq 2\bar\epsilon +P(\rho_{\bar\epsilon}).
\end{equation*}
Since for every $\bar \epsilon >0$  we can obtain the above inequality, letting $\bar \epsilon\rightarrow 0$, we have that
$P(\Phi)=\lim_{\bar\epsilon \rightarrow 0}P(\rho_{\bar \epsilon})$.
\end{proof}

\begin{proof}[Proof of Proposition \ref{defpressure}]
Since $\Phi$ is asymptotically additive, for a fixed $\bar \epsilon>0$, there exist a continuous function $\rho_{\bar \epsilon}$ and   $N_{\bar \epsilon} \in \N$ such that for every $n>N_{\bar \epsilon}$ we have that
\begin{equation*}
(S_n\rho_{\bar\epsilon})(\omega)-2n\bar{\epsilon}\leq  \phi_n(\omega)  \leq (S_n\rho_{\bar\epsilon})(\omega)+2n\bar\epsilon.
\end{equation*}
Therefore, for every $n>N_{\bar \epsilon}$ we have that
\begin{eqnarray*}\label{key1}
\exp(-2n\bar\epsilon)  \sum_{\omega \in \text{Fix}(\sigma^n)}\exp((S_n\rho_{\bar\epsilon})(\omega))\leq &\\
 \sum_{\omega \in \text{Fix}(\sigma^n)}\exp(\phi_n(\omega))
\leq
\exp(2n\bar\epsilon)  \sum_{\omega \in \text{Fix}(\sigma^n)}\exp((S_n\rho_{\bar\epsilon})(\omega)).
\end{eqnarray*}
Thus,
\begin{eqnarray*}\label{key1}
-2\bar\epsilon + \frac{1}{n} \log \left(  \sum_{\omega \in \text{Fix}(\sigma^n)}\exp((S_n\rho_{\bar\epsilon})(\omega))  \right)\leq &\\
 \frac{1}{n} \log \left( \sum_{\omega \in \text{Fix}(\sigma^n)}\exp(\phi_n(\omega)) \right)
\leq
2\bar\epsilon +  \frac{1}{n} \log \left( \sum_{\omega \in \text{Fix}(\sigma^n)}\exp((S_n\rho_{\bar\epsilon})(\omega))\right).
\end{eqnarray*}
Letting $n \to \infty$ and making use of Ruelle's representation of the pressure of a continuous function by means of periodic points as in equation \eqref{ruelle}  we obtain
\begin{equation*}
-2\bar\epsilon + P(\rho_{\bar\epsilon}) \leq \lim_{n \to \infty}  \frac{1}{n} \log \left( \sum_{\omega \in \text{Fix}(\sigma^n)}\exp(\phi_n(\omega )) \right)
\leq 2\bar\epsilon + P(\rho_{\bar\epsilon}).
\end{equation*}
Since for every $\bar \epsilon >0$  we can obtain the above inequality, letting $\bar \epsilon\rightarrow 0$ and making use of Lemma \ref{approx} we have that
\begin{equation*}
P(\Phi)= \lim_{\bar\epsilon \to 0}  P(\rho_{\bar\epsilon}) \leq \lim_{n \to \infty}  \frac{1}{n} \log \left( \sum_{ \omega \in \text{Fix}(\sigma^n)}\exp(\phi_n(\omega)) \right)
\leq \lim_{\bar\epsilon \to 0}  P(\rho_{\bar\epsilon})= P(\Phi).
\end{equation*}
\end{proof}

\section{Weak Gibbs measure for asymptotically additive sequences} \label{wgaas}
The notion of weak Gibbs measure can be extended to the setting of asymptotically additive sequence.
\begin{defi}
A probability measure $\mu$ is called a \emph{weak Gibbs measure} for the  asymptotically additive sequence $\Phi=(\phi_n)_n$ on
$\Sigma$ if there exists a  sequence of positive real numbers  $K(n)$ satisfying
\[\lim_{n \to \infty} \frac{ \log K(n)}{n}=0,\]
such that, for every $n \in \N$, every cylinder $C_{i_1 \dots i_n}$ and every  $\omega \in
C_{i_1 \dots i_n}$ we have
\begin{equation*}
\frac{1}{K(n)} \leq \frac{\mu(C_{i_1 \dots i_{n}})}{\exp(\phi_n(\omega) -n P(\Phi))}\leq K(n),\end{equation*}
where $P(\Phi)$ is the topological pressure for $\Phi$.
\end{defi}
\begin{rem}
It was shown by Barreira \cite{b2}  that if $\Phi$ is an almost additive sequence of continuous functions with tempered variation,  then there exists an ergodic weak Gibbs measure, not necessary invariant though (note that in virtue of Remark
\ref{reg} the regularity assumption is not really needed). This result generalizes that of Kesseb\"ohmer. Indeed,  if $\phi$ is a
continuous potential on $\Sigma$ then  $\Phi= (S_n(\phi))_n$ is an almost additive sequence. Moreover,  $\Phi$ has tempered variation if and only if $\phi$ is a medium varying function.
\end{rem}

In this section we prove that every weak Gibbs measure for an asymptotically additive sequence is a Gibbs measure for, another,  asymptotically additive sequence. Simplifying, therefore, the study of such measure.

\begin{teo} \label{main-weak}
If $\mu$ is a weak Gibbs measure for  an asymptotically additive sequence $\Phi:=(\phi_n)_n$ on a topologically mixing Markov shift $(\Sigma, \sigma)$ then
there exists an asymptotically additive sequence $\Psi:=(\psi_n)_n$ on $\Sigma$ such that the measure $\mu$ is Gibbs with respect to $\Psi$.
\end{teo}

\begin{proof}
Since the sequence $\Phi$ is asymptotically additive,  for every  $k \in \N$ there exists a continuous function $\rho_k:\Sigma \to \R$ such that
\[ \limsup_{n \to \infty} \frac{1}{n} \| \phi_n - S_n \rho_k \| \leq \frac{1}{k}. \]
Let us consider the following sequence of functions $ \Psi=(\psi_n)$, where $\psi_n:\Sigma \to \R$ is defined for $\omega \in C_{i_1 \dots i_n}$ by $\psi_n(\omega):= \log \mu(C_{i_1 \dots i_n})$.  Note that for every $\omega \in \Sigma$ we have that
\begin{equation} \label{psi_phi}
- \log K(n) \leq \psi_n(\omega) - \phi_n(\omega) +nP(\Phi)  \leq \log K(n).
\end{equation}
Indeed, since the measure $\mu$ is weak Gibbs for $\Phi$  equation \eqref{psi_phi} simply follows applying logarithm to the following inequalities,
\begin{equation*}
\frac{1}{K(n)} \leq \frac{\mu(C_{i_1 \dots i_n})}{\exp(-nP(\Phi) + \phi_n(\omega))}  \leq K(n).
\end{equation*}
Moreover,  the sequence  $\Psi=(\psi_n)_n$ is asymptotically additive. Indeed, it follows from equation \eqref{psi_phi} and the definition of weak Gibbs measure that for every $k \in \N$ the continuous function $(\rho_k +P(\Phi)) :\Sigma \to \R$, defined by $(\rho_k +P(\Phi))(\omega)= \rho_k(\omega) +P(\Phi)$,  is such that
\begin{eqnarray*}
 \limsup_{n \to \infty} \frac{1}{n}   \| \psi_n - S_n (\rho_k +P(\Phi)) \| \leq &\\  \limsup_{n \to \infty} \frac{1}{n}  \| \psi_n - \phi_n  + n P(\Phi)\|  +   \limsup_{n \to \infty} \frac{1}{n}  \| \phi_n - S_n \rho_k \|  \leq &\\
 \limsup_{n \to \infty} \frac{\log K(n)}{n} + \frac{1}{k} = \frac{1}{k}.
\end{eqnarray*}
It now follows from Proposition \ref{defpressure} and the definition of weak Gibbs measure that
\begin{eqnarray*}
P(\Psi)= \lim_{n \to \infty} \frac{1}{n} \log \sum_{ \omega \in \text{Fix}(\sigma^n)} \exp(\psi_n(\omega)) =&\\
\lim_{n \to \infty} \frac{1}{n} \log \sum_{\omega \in \text{Fix}(\sigma^n)} \mu(C_{i_1 \dots i_n} )\leq
 \lim_{n \to \infty} \frac{1}{n}  \log \sum_{ \omega \in \text{Fix}(\sigma^n)} \exp(-n P(\Phi)+ \phi_n (\omega)) K(n)  =&\\
\lim_{n \to \infty} \frac{\log K(n)}{n} -P(\Phi)  +  \lim_{n \to \infty} \frac{1}{n}  \log \sum_{\omega \in \text{Fix}(\sigma^n)} \exp(\phi_n (\omega)) =&\\
\lim_{n \to \infty} \frac{\log K(n)}{n} -P(\Phi) +P(\Phi)=0.
\end{eqnarray*}
And on the other hand
\begin{eqnarray*}
P(\Psi)= \lim_{n \to \infty} \frac{1}{n} \log \sum_{\omega \in \text{Fix}(\sigma^n)} \exp(\psi_n(\omega)) =&\\
\lim_{n \to \infty} \frac{1}{n} \log \sum_{\omega \in \text{Fix}(\sigma^n)} \mu(C_{i_1 \dots i_n} )\geq
 \lim_{n \to \infty} \frac{1}{n}  \log \sum_{\omega \in \text{Fix}(\sigma^n)} \exp(-nP(\Phi) + \phi_n (\omega)) K(n)^{-1}  =&\\
\lim_{n \to \infty} -\frac{\log K(n)}{n}-P(\Phi) +  \lim_{n \to \infty} \frac{1}{n}  \log \sum_{\omega \in \text{Fix}(\sigma^n)} \exp( \phi_n (\omega))=&\\
\lim_{n \to \infty} -\frac{\log K(n)}{n} -P(\Phi) +P(\Phi)=0.
\end{eqnarray*}
Therefore  $P(\Psi)=0$. Thus, for every $n \in \N$,  every cylinder $C_{i_1 \dots i_n}$ and every  $\omega \in
C_{i_1 \dots i_n}$ we have
\begin{equation*}
\frac{\mu(C_{i_1 \dots i_{n}})}{\exp(\psi_n(\omega) -n P(\Psi))} =1.
\end{equation*}
Therefore,  the measure $\mu$ is Gibbs with respect to $\Psi$.
\end{proof}

\begin{coro}
Every weak Gibbs measure $\mu$ for a continuous  potential $\phi$ on $\Sigma$ is a Gibbs measure for an
\emph{asymptotically additive} sequence $\Psi:=(\psi_n)_n$ on $\Sigma$.
\end{coro}

\begin{proof}
The result follows considering $\Phi=(S_n\phi)_n$ in Theorem \ref{main-weak}.
\end{proof}

Recall that an almost additive sequence $\Phi$ is an asymptotically additive sequence and thus it has
tempered variation (see Remark \ref{reg} and \cite[Proposition 2.1]{zzc}). Therefore there always exists a weak Gibbs measure for $\Phi$.

\begin{coro}
Let $\mu$ be a weak Gibbs measure for an almost additive sequence $\Phi:=(\phi_n)_n$. Then there exists an asymptotically additive sequence $\Psi:=(\psi_n)_n$ such that the measure $\mu$ is Gibbs with respect to $\Psi$.
\end{coro}

\begin{proof}
The result is contained in Theorem \ref{main-weak}.
\end{proof}

\begin{rem}[Regularity] \label{rem:regularity} Note that the sequence $\Psi=(\psi_n)_n$ constructed in Theorem \ref{main-weak} is such that every function $\psi_n:\Sigma \to \R$ is locally constant on cylinders of length $n$.
\end{rem}

\begin{rem}
Let $\phi$ be a continuous potential for which there exists a Gibbs measure $\mu$. Then the sequence of functions  $\Psi$ in Theorem \ref{main-weak} is almost additive. Without loss of generality we can assume that $P(\phi)=0$. Recall that there exists a constant $C>0$ such that
\begin{equation*}
\frac{1}{C} \exp(S_{n+m}\phi(\omega)) \leq \mu(C_{i_1 \dots i_{n+m}}) \leq C \exp(S_{n+m}\phi(\omega)) .
\end{equation*}
Note that
\begin{eqnarray*}
\psi_{n+m}(\omega)  - \psi_n(\omega) -\psi_{m}(\sigma^n \omega) = \log \frac{\mu(C_{i_1 \dots i_{n+m}}) }{\mu(C_{i_1 \dots i_{n}}) \mu(C_{i_n \dots i_{n+m}}) }  \leq &\\
\log \frac{C \exp S_{n+m}\phi(\omega)}{C^{-2}\exp S_{n}\phi(\omega)\exp S_{m}\phi(\sigma^n \omega) } =3 \log C.
\end{eqnarray*}
Moreover,
\begin{eqnarray*}
\psi_{n+m}(\omega)  - \psi_n(\omega) -\psi_{m}(\sigma^n \omega) = \log \frac{\mu(C_{i_1 \dots i_{n+m}}) }{\mu(C_{i_1 \dots i_{n}}) \mu(C_{i_n \dots i_{n+m}}) }  \geq &\\
\log \frac{C^{-1} \exp S_{n+m}\phi(\omega)}{C^{2}\exp S_{n}\phi(\omega)\exp S_{m}\phi(\sigma^n \omega) } =-3 \log C.
\end{eqnarray*}
Therefore the sequence $\Psi$ is almost additive. Now, if the measure $\mu$ is weak Gibbs but not Gibbs then the sequences $\Psi$ is asymptotically additive but not almost additive.
\end{rem}

\begin{rem}
It should be stressed that the results in Theorem \ref{main-weak} and in its Corollaries the Gibbs condition is satisfied for a constant $C=1$, thus we obtain equalities instead of inequalities in the Gibbs definition.
\end{rem}

\begin{rem} \label{rem:paulo}
In \cite{vz} Varandas and Zaho studied Large deviations for weak Gibbs measures for non-additive sequences of potentials.  Theorem \ref{main-weak} shows that in their results it is only necessary to assume that the measure is indeed a Gibbs measure for the appropriate sequence.
\end{rem}

%
%

\section{Multifractal analysis}

In this section we apply several results on  dimension theory of non-additive sequences to the study of multifractal analysis of conformal dynamical systems. It is worth stressing that while  non-additive thermodynamic formalism was originally  developed to study dimension theory on non-conformal settings, it has interesting and powerful applications in the conformal setting. Multifractal analysis is the study of level sets determined by (dynamically defined) local quantities. In general, the geometry of the level sets is rather complicated and in order to quantify its size Hausdorff dimension is used. Combining the results in Section \ref{wgaas} with those recently obtained by Barreira, Cao and Wang \cite{bcw} on multifractal analysis of quotients of asymptotically additive sequences,  we obtain new results and recover old ones in multifractal analysis for one-dimensional uniformly and non-uniformly hyperbolic dynamical systems. The class of dynamical systems that we will study is the following,
\begin{defi} \label{maps}
Let $I=[0,1]$ and consider a finite  family $\{ I_i \}_{i=1}^k$, of disjoint closed intervals contained in I. A map $T:\cup_{i=1}^kI_i \to I $ is an \emph{expanding Markov map} if
\begin{enumerate}
\item the map is $C^1$;
\item the map $T$ is Markov and it can be coded by a topologically mixing sub-shift of finite on the alphabet $\{1, \dots, k\}$;
\item for every $x \in   \cup_{i=1}^kI_i $ we have that $|T'(x)| \geq 1$ and there exists, at most, finitely many points for which $|T'(x)|=1$.
\end{enumerate}
\end{defi}
The \emph{repeller} of such a map is defined by
\[\Lambda:=\left\{x \in \cup_{i=1}^{\infty} I_i: T^n(x) \textrm{ is well defined for every } n \in \mathbb{N} \right\}.\]
The Markov structure assumed for expanding Markov maps $T$, allows for a good symbolic representation. This means that there exists a topologically mixing sub-shift of finite type $(\Sigma, \sigma)$ defined on the alphabet $\{1, \dots, k \}$
with the following property: there exists an homeomorphism, the natural projection,  $\pi :\Sigma \to \Lambda$ such that $\pi \circ \sigma = T  \circ  \pi$.

When coding a dynamical system with a symbolic one it is often desirable to translate the original metric properties onto the symbolic space. A standard way to do it is to associate to each symbolic cylinder of length $n$ the diameter of its projection.
Let $I(i_1, \dots i_n) := \pi (C_{i_1 \dots i_n})$, with a slight abuse of notation we call that set a \emph{cylinder} of length $n$ for $T$. For every $n \in \N$ we define the function $D_n: \Sigma \to \R$ by  $D_n(\omega):= \text{diam} ( I(i_1, \dots i_n))$, where $I(i_1, \dots i_n)$ is the cylinder of length $n$ in $[0,1]$ containing the point $\pi(\omega)$. We denote the sequence of functions obtained in this way by  $D:=(D_n)_n$. Urba\'nksi \cite{u} and also Jordan and Rams \cite{jr} noted (using a different language though)  that,

\begin{lema} [Urba\'nksi, Jordan-Rams] \label{ujr}
If $T$ is a expanding Markov map then the sequence of functions  $\log D:=( \log D_n)_n$ is asymptotically additive.
\end{lema}

\begin{proof}
From \cite[Lemma 1]{jr} we have that if $\gamma:= \log|T'| \circ \pi$ (that is the symbolic representation of the logarithm of the derivative)  then there exists a sequence of positive real numbers $M(n)$ with $\lim_{n \to \infty} M(n)=0$ such that for every $\omega \in \Sigma$ we have
\begin{equation*}
e^{-n M(n)}  \leq D_n(\omega
) e^{-S_n \gamma(\omega)} \leq e^{n M(n)}.
\end{equation*}
Thus,
\begin{equation*}
\frac{1}{n} \| \log D_n - S_n \gamma \| \leq M(n)
\end{equation*}
Since $\lim_{n \to \infty} M(n)=0$ the result follows.
\end{proof}

We now define some of he local quantities that we will consider. Let $\mu$ be a probability measure on $[0,1]$. The \emph{pointwise dimension} of $\mu$ at the point $x \in (0,1)$ is defined by
\begin{equation*}
d_{\mu}(x):= \lim_{r \to 0} \frac{\log \mu((x-r,x+r))}{\log r},
\end{equation*}
whenever the limit exists. This function describes the power law behaviour of the measure for small intervals and has been extensively studied over the last decade (see for example \cite{b4}). For a weak Gibbs measure $\mu$ the pointwise dimension can be computed at a symbolic level (see \cite[Lemma 8]{jr}).

\begin{lema} \label{dp} Let $T$ be an expanding Markov map and $\mu$ a weak Gibbs measure
corresponding to the continuous potential $\phi$. Denote by $\overline{\mu}:= \mu \circ \pi$, then whenever the pointwise
dimension of $\mu$ exists we have
\begin{equation*}
d_{\mu}(x):= \lim_{r \to 0} \frac{\log \mu((x-r,x+r))}{\log r}=-\lim_{n \to \infty} \frac{\log \overline{\mu}(C_{i_1 \dots i_n})}{\log D_n(x)},
\end{equation*}
where $C_{i_1 \dots i_n}$ is the cylinder of length $n$ containing $\pi^{-1}x$.
\end{lema}

We can now state  our result in dimension theory. We  denote the Hausdorff dimension of a set by
$\dim_H(\cdot)$ (see \cite{f2} for  definition and properties). Recall from Section \ref{wgaas} that if $\mu_i$ is a weak Gibbs measure then we can define an
asymptotically additive sequence $\Psi^i:=(\psi^i_n)_n$ by $\psi^i_n(x)=\log \mu_i(C_{i_1 \dots i_n}(x))$.

\begin{teo} \label{mfa}
Let $T$ be an expanding Markov map and $(\mu_i)_{i=1}^r$ be non-atomic invariant weak Gibbs measures. Let $\overline{\alpha}= (\alpha_1, \alpha_2, \dots , \alpha_r) \in \R^r$ and consider the level sets defined by
\begin{equation*}
J(\overline{\alpha}):= \left\{ x \in [0,1] : \left( d_{\mu_1}(x), d_{\mu_2}(x) ,  \dots , d_{\mu_r}(x) \right)=\overline{\alpha}\right\}
\end{equation*}
Then,
\begin{eqnarray*}
\dim_H J(\overline{\alpha})= &\\ \sup \left\{ \frac{h(\nu)}{\int \log|T'| d\nu} : \lim_{n \to \infty}\frac{1}{n} \int \frac{\psi_n^i}{\log D_n} d \nu= \alpha_i, \text{ for all } i \in \{1, \dots ,r \} , \nu \in \M \right\}.
\end{eqnarray*}
\end{teo}

\begin{proof}
Since $\mu_i$ is weak Gibbs we have that  for every $i \in \{1, \dots , r\}$ the sequences $\Psi^i:=(\psi_n^i)$ are  asymptotically additive (see Theorem \ref{main-weak}). Also since  $T$ is  an  expanding Markov map  the sequence $\log D$ is asymptotically additive (see Lemma \ref{ujr}). Moreover,  by Lemma \ref{dp}, we have
\[d_{\mu_i}(x)=-\lim_{n \to \infty} \frac{\psi^i_n(\omega)}{\log D_n(\omega)},\]
where $\pi(\omega)=x$. Thus, the pointwise dimension of $\mu_i$ at the point $x \in [0,1]$ is the quotient of two asymptotically additive sequences. The result now follows from \cite[Theorem 1]{bcw}.
\end{proof}

\begin{rem}
Note that what is relevant in the proof of Theorem \ref{mfa} is that the pointwise dimension of weak Gibbs measures can be computed as the limit of the quotient of two asymptotically additive sequences. The same type of result can be obtained considering, instead of pointwise dimension,  local entropies, Birkhoff averages, Lyapunov exponents or quotients of Birkhoff averages.
\end{rem}

\end{document}